\documentclass[10pt,a4paper]{article}

\usepackage[british]{babel}
\usepackage[utf8]{inputenc}
\usepackage{datetime}
\newdateformat{dateformatone}{\monthname[\THEMONTH] \THEDAY, \THEYEAR}

\usepackage{tabularx}

\usepackage{amssymb,amsmath,amsthm,thmtools}
\usepackage{float}
\numberwithin{equation}{section}
\usepackage{mathtools}	
\usepackage{accents}
\usepackage{multirow}
\usepackage{graphicx}
\usepackage{amscd}
\usepackage{epic, eepic}
\usepackage{url}
\usepackage{xcolor}
\usepackage{todonotes}
\usepackage{enumerate}
\usepackage{enumitem}
\usepackage{comment}
\usepackage{hhline}
\usepackage{bbm}
\usepackage[hidelinks]{hyperref}
    \addto\extrasbritish{%
    }

\usepackage{tikz-cd}

\newcommand{\ff}{\mathbb{F}}

\newcommand{\zz}{\mathbb{Z}}
\newcommand{\qq}{\mathbb{Q}}

\DeclareMathOperator{\Sp}{Sp}
\DeclareMathOperator{\pf}{pf}

\newcommand{\symdiff}{\mathop{\bigtriangleup}}
\newcommand{\twovec}[2]{\begin{pmatrix}#1 \\ #2\end{pmatrix}}
\newcommand{\twomatr}[4]{\begin{pmatrix}#1 & #2\\ #3 & #4\end{pmatrix}}

 \makeatletter
 \@addtoreset{equation}{section}
 \makeatother

 \makeatletter
 \@addtoreset{table}{section}
 \makeatother
 
\declaretheorem[style=plain,name=Theorem,numberwithin=section]{theorem}
\declaretheorem[style=plain,name=Proposition,sibling=theorem]{proposition}
\declaretheorem[style=plain,name=Lemma,sibling=theorem]{lemma}
\declaretheorem[style=plain,name=Problem,sibling=theorem]{problem}
\declaretheorem[style=plain,name=Corollary,sibling=theorem]{corollary}

\declaretheorem[style=plain,name=Construction,sibling=theorem]{construction}
\declaretheorem[style=plain,name=Observation,sibling=theorem]{observation}
\declaretheorem[style=plain,name=Observations,sibling=theorem]{observations}

\declaretheorem[style=definition,name=Definition,sibling=theorem]{definition}
\declaretheorem[style=definition,name=Notation,sibling=theorem]{notation}
\declaretheorem[style=definition,name=Conjecture,sibling=theorem]{conjecture}

\declaretheorem[style=definition,name=Remark,sibling=theorem]{remark}

    \makeatletter
    \providecommand*{\cupdot}{%
      \mathbin{%
        \mathpalette\@cupdot{}%
      }%
    }
    \newcommand*{\@cupdot}[2]{%
      \ooalign{%
        $\m@th#1\cup$\cr
        \sbox0{$#1\cup$}%
        \dimen@=\ht0 %
        \sbox0{$\m@th#1\cdot$}%
        \advance\dimen@ by -\ht0 %
        \dimen@=.5\dimen@
        \hidewidth\raise\dimen@\box0\hidewidth
      }%
    }
    
    \providecommand*{\bigcupdot}{%
      \mathop{%
        \vphantom{\bigcup}%
        \mathpalette\@bigcupdot{}%
      }%
    }
    \newcommand*{\@bigcupdot}[2]{%
      \ooalign{%
        $\m@th#1\bigcup$\cr
        \sbox0{$#1\bigcup$}%
        \dimen@=\ht0 %
        \advance\dimen@ by -\dp0 %
        \sbox0{\scalebox{2}{$\m@th#1\cdot$}}%
        \advance\dimen@ by -\ht0 %
        \dimen@=.5\dimen@
        \hidewidth\raise\dimen@\box0\hidewidth
      }%
    }
    \makeatother

\title{Code-based $[3,1]$-avoiders in finite affine spaces $\mathrm{AG}(n,2)$}

\author{Benedek Kovács\thanks{ELTE Linear Hypergraphs Research Group, Eötvös Loránd University, Budapest, Hungary. The author is supported by the EKÖP-24 University Excellence Scholarship Program of the Ministry for Culture and Innovation from the source of the National Research, Development and Innovation Fund.
		E-mail: {\tt benoke981@gmail.com}}
}
\date{\dateformatone\today}

\begin{document}
	
	\maketitle
	
	\begin{abstract}
        The author, together with Nagy, studied the following problem on unavoidable intersections of given size in binary affine spaces. Given an $m$-element set $S\subseteq \ff_2^n$, is there guaranteed to be a $[k,t]$-flat, that is, a $k$-dimensional affine subspace of $\ff_2^n$ containing exactly $t$ points of $S$? Such problems can be viewed as generalizations of the cap set problem over the binary field. They conjectured that for every fixed pair $(k,t)$ with $k\ge 1$ and $0\le t\le 2^k$, the density of values $m\in \{0,...,2^n\}$ for which a $[k,t]$-flat is guaranteed tends to $1$.
        
        In this paper, motivated by the study of the smallest open case $(k,t)=(3,1)$, we present explicit constructions of sets in $\ff_2^n$ avoiding $[k,1]$-flats for exponentially many sizes. These sets rely on carefully constructed binary linear codes, whose weight enumerators determine the size of the construction.
		
		\textit{Keywords}: affine space, linear code, unavoidable, weight enumerator
	\end{abstract}
    
\section{Introduction}

Consider the points of the $n$-dimensional affine space $\mathrm{AG}(n,2)$ over the binary field $\ff_2$, which also correspond to the elements of the vector space $\ff_2^n$.

Let $S\subseteq \mathrm{AG}(n,2)$ be a point set of a fixed cardinality $m$. We investigate the intersections of $S$ with the $k$-dimensional affine subspaces of $\ff_2^n$, referred to as \textit{$k$-flats}. We say that $S$ \textit{induces} a $[k,t]$-flat if there exists a $k$-flat $F_k\subseteq \ff_2^n$ such that $|S\cap F_k|=t$. Otherwise we say that $S$ is a \textit{$[k,t]$-avoider}.

Our question is the following: for what quadruples $(n,m;k,t)$ does it hold that every $S\subseteq \ff_2^n$ of size $m$ induces a $[k,t]$-flat? If this is true, we write $[n,m]\to [k,t]$. The question was posed by the author and Nagy \cite{KN25} as a $q$-analogue of a graph-theoretical problem by Erdős, Füredi, Rothschild and T. Sós \cite{EFRS99}. Note that a $[k,0]$-avoider is also called a \textit{blocking set} with respect to $k$-dimensional subspaces, which is an extensively covered topic in the projective and affine settings (see \cite{BSS11} for a survey). Also note that the corresponding question in the ternary field $\ff_3$ contains the well-known \textit{cap set problem} \cite{EG17}, asking for the largest size of a subset in $\ff_3^n$ without a full line. This would correspond to $k=1$ and $t=3$.

We defined the $(n;k,t)$\textit{-spectrum} $\Sp(n,k,t)$, and its \textit{density}, $\rho(n;k,t)$, by
\begin{equation*}
\begin{split}
\Sp(n;k,t) & :=\{0\le m\le 2^n: [n,m]\to [k,t] \},\\
\rho(n;k,t) & :=\frac{|\Sp(n;k,t)|}{2^n}.
\end{split}
\end{equation*}

Here, and throughout this paper, we assume that $(n,k,t)$ is a triple of integers with $n\ge 1$, $1\le k\le n$ and $0\le t\le 2^k$.

We presented the following main conjecture:
\begin{conjecture}[Kovács, Nagy \cite{KN25}]\label{con:maincon} For every $k\ge 1$ and $0\le t\le 2^k$, we have
$$\lim_{n\to \infty} \rho(n;k,t)=1.$$		
\end{conjecture}

The conjecture was shown true for $t\in \{0, 2^{k-1}, 2^k\}$ when $k\ge 3$, and for all $t$ when $k\in \{1,2\}$. The main result used for this (apart from the basic cases) was a result similar to Szemerédi's Cube Lemma (see \cite[Corollary 2.1]{Setyawan98}), which also follows from a work of Bonin and Qin \cite{BQ00}, stating that for $m\ge \frac52\cdot 2^{n\left(1-\frac{1}{2^{k-1}}\right)}$, we have $[n,m]\to [k,2^k]$. The proof proceeded by picking the value $\mathbf{d}\in \ff_2^n$ that appears the most commonly as a difference of two elements in $S$, and applying the statement inductively for the quotient space $\ff_2^n/\left\langle \mathbf{d}\right\rangle$.

We also showed that $\liminf\rho(n;k,t)>0$ when $t$ is of the form $2^a$ or $3\cdot 2^a$ for some integer $1\le a\le k-2$. The proof of the latter result was more involved, using the notion of \textit{additive energy} defined by Tao and Vu \cite{TV06}, and bounding sizes of hypercube cuts based on a theorem of Hart \cite{H75}. However, \autoref{con:maincon} stays open in these cases too.

Improving on an argument by Guruswami \cite{Guruswami11}, we also used the probabilistic method to obtain so-called \textit{$[k,c]$-evasive} sets, defined as sets in $\ff_2^n$ intersecting every $k$-flat in at most $c$ points. Our result was the following:

\begin{theorem}\label{thm:ujgur}
	Let $k,c\in \zz_{>0}$ with $c\ge k+1$. Then for some $K>0$, there exists a $[k,c]$-evasive set in $\ff_2^n$ of size at least $K\cdot 2^{n\left(1-\frac{k}{c}\right)}$ when $n$ is sufficiently large.
\end{theorem}

As $[k,c]$-evasive sets are $[k,t]$-avoiders for every $t>c$, this theorem can easily be used to show that the number of sizes $m$ outside $\Sp(n;k,t)$ is exponentially large for each pair $(k,t)$ with $k\ge 3$ and $(k,t)\ne (3,4)$. For example, since the complement of any $[3,6]$-evasive set is a $[3,1]$-avoider, there exist $[3,1]$-avoiders of all sizes $2^n-K\cdot 2^{n/2}\le m\le 2^n$ by \autoref{thm:ujgur}. However, as these $[k,t]$-avoiders are obtained with a random method, there is no systematic way to describe them.

The research presented in this paper mainly concerns the smallest open case of \autoref{con:maincon}, which is $(k,t)=(3,1)$. Our main result is a construction of an explicit family of $[3,1]$-avoiders that still exhibit exponentially many different sizes. We call this the \textit{code-based construction}, which takes an arbitrary binary linear code as an input, and derives a $[k,1]$-avoider (for fixed $k\ge 3$) whose size only depends on the weight enumerator of the code evaluated at certain inputs.

Let us make note of the following definitions.

\begin{definition}[Weight of a codeword]
In a binary linear code $C$, the \textit{weight} of a codeword $\mathbf{x}\in C$ is the number of coordinates of $\mathbf{x}$ equal to $1$.
\end{definition}

\begin{definition}[Weight enumerator]
The \textit{weight enumerator} of a binary code $C\subseteq \ff_2^{\ell}$ is a bivariate polynomial $W_C$ defined by
$$W_C(x,y)=\sum_{w=0}^{\ell} A_w x^w y^{\ell-w},$$

where $\ell$ is the length of $C$, and $A_w$ is the number of codewords in $C$ having weight $w$ for each $0\le w\le \ell$.
\end{definition}

For further details on the coding-theoretical terminology used in the paper, see Subsection \ref{subsec:codingtheory}. We now state our main results.

\begin{restatable}{theorem}{restateavoiderfromcode}\label{thm:avoiderfromcode}
	Fix integers $k\ge 3$ and $\ell\ge 1$. Take a binary linear code $C\le \ff_2^{\ell}$. Then for $n=\ell(k-1)$, there exists a $[k,1]$-avoider $S$ in $\ff_2^n$ for which $|S|=2^n-W_C(1, 2^{k-1}-1)$.
\end{restatable}

Here the set $S$ is obtained using the code-based construction, which is detailed in Construction \ref{con:codebased}.

\begin{restatable}{theorem}{restateexpmany}\label{thm:expmany}
	Suppose that $n\ge 4$ is an even integer. Then the number of different sizes of $[3,1]$-avoiders in $\ff_2^n$ given by the code-based construction is at least $\binom{\left\lfloor(n-4)/4\right\rfloor}{\left\lfloor(n-4)/8\right\rfloor}^{1/3}\ge c\cdot\frac{2^{n/12}}{\sqrt{n}}$ for some absolute constant $c>0$.
\end{restatable}

Note that here we got at least $\approx \frac{c}{\sqrt{n}}\cdot 1.059^n$ different sizes, but as we will remark later, by the MacWilliams identity, the number of different sizes obtainable by this method is bounded above by $c\cdot n\cdot (\sqrt{3})^n\approx c\cdot n\cdot 1.732^n$.

The outline of the paper is the following. In Section \ref{sec:notation}, we set the basic notation and definitions. In Section \ref{sec:k1avoid}, we give examples of $[k,1]$-avoiders that are built from affine subspaces using union and symmetric difference operations. We then describe a construction that uses the independent sets of a simple graph or a hypergraph. Then in Section \ref{sec:avoiderfromcode}, we present our code-based construction, which provides a proof of Theorem \ref{thm:avoiderfromcode}. We then present a coding-theoretical limitation that puts an upper bound on the number of possible sizes of $[k,1]$-avoiders resulting from this construction. Finally, Section \ref{sec:expmany31avoid} is devoted to the proof of Theorem \ref{thm:expmany}. Here we introduce two transformations of linear codes that each increase the length of a code by two, and we apply a sequence of such transformations to obtain a family of exponentially many codes. Using a group-theoretical tool, we show that these codes give rise to $[3,1]$-avoiders of exponentially many sizes.

\section{Notation and preliminaries}\label{sec:notation}

By convention, elements of $\ff_2^n$ will be considered as $1\times n$ row vectors over $\ff_2$.

Similarly to \cite{KN25}, we define the \textit{$k$-profile} $\pf_k(S)$ of a set $S\subseteq \ff_2^n$ as the set of all possible intersection sizes of $S$ and a $k$-flat of $\ff_2^n$. So $t\not\in \pf_k(S)$ implies that $|S|\not\in \Sp(n;k,t)$. We say that the $k$-profile of $S$ is \textit{even} if all of its elements are even.

In a vector space, the direct sum of the two subspaces $A$ and $B$ is denoted by $A\oplus B$.

The complement of a set $S$ will be denoted by $\overline{S}$. Recall that the \textit{symmetric difference} of the $a$ sets $S_1, S_2, ..., S_a$, denoted by $\symdiff_{i=1}^a S_i$, consists of those elements in $\bigcup_{i=1}^a S_i$ that appear in an odd number of sets among $S_1, S_2, ..., S_a$. When $a=2$ (so for the symmetric difference of two sets), we can also use the notation $S_1\triangle S_2$.

We let $[n]=\{1, ..., n\}$. If $A\subseteq [n]$ is an arbitrary subset, one can define its characteristic vector $\mathbbm{1}_A$, which is a $0$-$1$ sequence of length $n$. If $\mathbf{v}=\mathbbm{1}_A$, then for every $i\in [n]$,
\begin{equation*}
\mathbf{v}_i=
\begin{cases}
1 & \text{if $i\in A$,}\\
0 & \text{if $i\not\in A$.}\\
\end{cases}
\end{equation*}

Similarly one can define the characteristic vector $\mathbbm{1}_S$ for a subset $S\subseteq \ff_2^n$, which is a $2^n$-long $0$-$1$ sequence with entries indexed by the vectors $\mathbf{x}\in \ff_2^n$.

In this paper, an \textit{alphabet} is a finite set $\Sigma$, whose elements are called \textit{letters}. For example, $\Sigma=\{\mathtt{a},\mathtt{b}\}$. If $r\ge 0$, then a sequence of $\sigma_1\sigma_2...\sigma_r$ of $r$ letters is called a \textit{word} of length $r$, and the set of such words is denoted $\Sigma^r$. We write $\Sigma^{*}=\bigcup_{r\ge 0}\Sigma^r$ for the set of all words over $\Sigma$.

We let $\zz_{\ge 0}$ and $\zz_{>0}$ denote the set of nonnegative and positive integers, respectively.

\subsection{Notions from coding theory}\label{subsec:codingtheory}
	
Let us recall the following basic notions from coding theory. A \textit{binary linear code} $C$ of length $\ell$ and dimension $d$ is a linear subspace of $\ff_2^{\ell}$ of dimension $d$. The elements of $C$ are called \textit{codewords}, and the \textit{size} of the code is the number of codewords, which is equal to $2^d$.

The dual code of $C$ is defined as $C^{\perp}=\{\mathbf{y}\in \ff_2^{\ell}: \mathbf{y} \mathbf{c}^T=0 ~
\forall \mathbf{c}\in C\}$, i.e., it is the orthogonal complement of the subspace $C\le \ff_2^{\ell}$ with respect to the standard inner product. Its dimension is $\ell-d$.

A \textit{generator matrix} $G$ for $C$ is a $d\times \ell$ matrix whose rows form a basis for $C$, and a \textit{parity check matrix} $H$ for $C$ is a generator matrix for $C^{\perp}$, so it is an $(\ell-d)\times \ell$ matrix. If $H$ is a parity check matrix for $C$, then for every $\mathbf{y}\in \ff_2^{\ell}$, we have $\mathbf{y}\in C \Leftrightarrow H\mathbf{y}^{T}=\mathbf{0}$.

We recall the MacWilliams identity, which gives a link between the weight enumerator of a code and its dual:

\begin{lemma}\label{lem:macwilliams}
[\protect{MacWilliams identity \cite[p. 127]{MacWilliamsBook}}] For a binary linear code $C$ of dimension $d$,
$$W_{C^{\perp}}(x,y)=\frac{1}{2^d} W_C(y-x,y+x).$$
\end{lemma}

\subsection{Notions from group theory}

The identity element of a group $G$ will be denoted $1_G$. The notation $H\le G$ means that $H$ is a subgroup of $G$. The subgroup of $G$ generated by the two subgroups $H_1,H_2\le G$ is denoted by $\langle H_1,H_2\rangle$.

The free group over the generator set $S$ will be denoted by $F(S)$, and the free product of two groups $G_1$, $G_2$ will be denoted by $G_1\ast G_2$.

\section{$[k,1]$-avoiders}\label{sec:k1avoid}

Our first fundamental example of $[k,1]$-avoiders (for $k\ge 2$) is obtained from the following basic observations:

\begin{observations}\label{obs:subspace}
\begin{enumerate}[label=(\alph*)]
\item If $H$ is an $(n-k+1)$-flat of $\ff_2^n$, then it has an even $k$-profile, so in particular $H$ is a $[k,1]$-avoider.
\item The union of an arbitrary collection of $[k,1]$-avoiders in $\ff_2^n$ is also a $[k,1]$-avoider.
\end{enumerate}
\end{observations}
\begin{proof}
\begin{enumerate}[label=(\alph*)]
 \item If the affine subspace $H\subseteq \ff_2^n$ has dimension $n-k+1$, then for any $k$-flat $F_k\subseteq \ff_2^n$, either $H\cap F_k=\emptyset$, or by the dimension theorem, the intersection is a flat of dimension $\dim H+\dim F_k-\dim \left\langle H, F_k\right\rangle=n+1-\dim \left\langle H, F_k\right\rangle$. Here $\left\langle H, F_k\right\rangle \subseteq \ff_2^n$ has dimension at most $n$, therefore the dimension of $H\cap F_k$ is at least 1, and so $|H\cap F_k|\ge 2$ is an even power of two.
 \item If $\{S_i: i\in I\}$ is a collection of $[k,1]$-avoiders in $\ff_2^n$, and $S=\cup_{i\in I} S_i$, then let $F_k$ be an arbitrary $k$-flat. We show that $|F_k\cap S|=1$ cannot hold. Assume to the contrary that it does. Then there exists $i\in I$ such that $F_k\cap S_i\ne\emptyset$. Take such an $i$ and let $|F_k\cap S_i|=c$. Since $S_i$ was a $[k,1]$-avoider, we must have $c\ge 2$. However $F_k\cap S_i\subseteq F_k\cap S$, therefore $|F_k\cap S|\ge c\ge 2$ as well, giving a contradiction.\qedhere
\end{enumerate}
\end{proof}

Putting together these two observations, we obtain the following:

\begin{corollary}\label{basic_avoiders}
For an integer $k\ge 2$, let $S\subseteq \ff_2^n$ be an arbitrary union of $(n-k+1)$-flats. Then $S$ is a $[k,1]$-avoider.
\end{corollary}

For example, any union of $(n-2)$-flats (\textit{quarter-spaces}) satisfies the $[3,1]$-avoider property. Sets obtained from Corollary \ref{basic_avoiders} will be called \textit{basic} $[k,1]$-avoiders.

The following observations help us obtain further $[k,1]$-avoiders.

\begin{observations}\label{obs:symdiff_subspaces}
  \begin{enumerate}[label=(\alph*)]
  \item Let $k,a\ge 2$ be integers and $S_1, S_2, ..., S_a\subseteq \ff_2^n$ be sets with even $k$-profiles. Then $\symdiff_{j=1}^a S_j$ also has an even $k$-profile.
  \item The symmetric difference of a finite collection of $(n-k+1)$-flats in $\ff_2^n$ is a $[k,1]$-avoider.
  \item Any set $S\subseteq \ff_2^n$ obtained in the form $S=\cup_{i=1}^b T_i$ (for $b\ge 0$), where each $T_i$ is a symmetric difference of a finite collection of $(n-k+1)$-flats, is a $[k,1]$-avoider.
  \end{enumerate}
\end{observations}
\begin{proof}
\begin{enumerate}[label=(\alph*)]
 \item For an arbitrary $k$-flat $F_k\subseteq \ff_2^n$, we have $F_k\cap (\symdiff_{j=1}^a S_j)=\symdiff_{j=1}^a (F_k\cap S_j)$. Here observe that a finite symmetric difference of even-cardinality sets has even cardinality too.
 \item By part (a) and Observation \ref{obs:subspace}(a), the symmetric difference has an even $k$-profile and so is a $[k,1]$-avoider.
 \item This follows from part (b) and Observation \ref{obs:subspace}(b). \qedhere
\end{enumerate}
\end{proof}

\subsection{Obtaining $[k,1]$-avoiders from hypergraphs}

As the following Proposition shows, a basic $[k,1]$-avoider in $\ff_2^n$ can be constructed from the independent sets of any hypergraph on $n$ vertices.

Let $H=(V,E)$ be a hypergraph with $V=[n]$ and $E\subseteq \mathcal{P}(V)$. An \textit{independent set} of $H$ is a subset $W\subseteq V$ such that no edge $e\in E$ is fully contained in $W$. Let $i(H)$ denote the number of independent sets of $H$. (If $H=G$ is a graph then $i(G)$ is also called the Fibonacci number of $G$, see \cite{PT82}.)

The \textit{rank} of a hypergraph $H$ is the maximum cardinality of a hyperedge of $H$, and is denoted by $r(H)$.

\begin{proposition}\label{prop:hypergraph_example}
	If $H=(V,E)$ is a hypergraph on $V=[n]$ with $r(H)\le k-1$, where $k\ge 3$, then there exists a $[k,1]$-avoider $S_H\subseteq \ff_2^n$ with $|S_H|=2^n-i(H)$.
\end{proposition}
\begin{proof}
	For each hyperedge $e\in E$, define $F_e=\{\mathbf{x}\in \ff_2^n: \mathbf{x}_v=1 ~\forall v\in e\}$, which is an $(n-c)$-flat if $|e|=c$, and every such $(n-c)$-flat can be decomposed as a union of $(n-k+1)$-flats.
	
	Now let $S_H\subseteq \ff_2^n$ be defined as $S_H=\bigcup_{e\in H} F_e$. Then $S_H$ is equal to a union of $(n-k+1)$-flats, so is a basic $[k,1]$-avoider. We can also see that $\mathbf{x}\not\in S_H$ if and only if $\mathbf{x}$ is a characteristic vector for an independent set of $H$. Therefore, $|S_H|=2^n-i(H)$.
\end{proof}

\begin{corollary}
If $G$ is a graph on $n$ vertices, then there exists a $[3,1]$-avoider $S_G\subseteq \ff_2^n$ with $|S_G|=2^n-i(G)$.
\end{corollary}
\begin{proof}
We apply Proposition \ref{prop:hypergraph_example} with $k=3$.
\end{proof}

Finding the number of possible values of $i(G)$ for an $n$-vertex graph $G$ is a question of independent interest as well. For a detailed presentation of the author's results together with Nagy on this problem, see \cite{indep_sets}.

\subsection{Reed-Muller code perspective}

We may look at the $[k,1]$-avoiders obtained in Observations \ref{obs:symdiff_subspaces} from a coding-theoretical perspective. The \textit{binary Reed-Muller code} $\mathrm{RM}(n,r)$ of length $2^n$ and order $r$ is generated by the characteristic vectors of $(n-r)$-flats of $\ff_2^n$ (see \cite[Chapter 13]{MacWilliamsBook}). Let us take $r=k-1$. A vector $\mathbf{x}\in \ff_2^{2^n}$ is a codeword of $\mathrm{RM}(n,k-1)$ if and only if it can be expressed as the sum of the characteristic vectors (over $\ff_2$) of some $(n-k+1)$-flats, say $\mathbf{x}=\sum_{i=1}^c \mathbf{x}^{(i)}$ where $c\ge 0$, and for each $i$, $\mathbf{x}^{(i)}=\mathbbm{1}_{F_i}$ for an $(n-k+1)$-flat $F_i\subseteq \ff_2^n$. But this is equivalent to saying that $\mathbf{x}=\mathbbm{1}_S$ for the set $S=\symdiff_{i=1}^c F_i$, therefore codewords of $\mathrm{RM}(n,k-1)$ are precisely the characteristic vectors of $[k,1]$-avoiders from Observation \ref{obs:symdiff_subspaces}(b).

For a finite number of vectors $\mathbf{y}^{(1)}, ..., \mathbf{y}^{(b)}\in \ff_2^n$, one can define their \textit{elementwise disjunction} $\mathbf{y}=\bigvee_{i=1}^b \mathbf{y}^{(i)}$, where $\mathbf{y}_j=1$ if and only if $(\mathbf{y}^{(i)})_j=1$ for at least one value $i$. Then the characteristic vectors of $[k,1]$-avoiders in Observation \ref{obs:symdiff_subspaces}(c) are precisely the vectors arising as a disjunction of some codewords of $\mathrm{RM}(n,k-1)$. Our question is then:

\begin{problem}
Determine the possible weights of the elementwise disjunction of a collection of codewords in $\mathrm{RM}(n,k-1)$.
\end{problem}

For $r=1$ this is an easy question to investigate using the affine space itself: a symmetric difference of a finite number of $(n-1)$-flats (halfspaces) is either $\emptyset$, the full space $\ff_2^n$ or an $(n-1)$-flat. So in this case, a union of such objects $S=\cup_{i=1}^b T_i$ (as in Observation \ref{obs:symdiff_subspaces}(c)) satisfies that $\overline{S}=\cap_{i=1}^b \overline{T_i}$, where $\overline{S}$ can be $\emptyset$, $\ff_2^n$ or an $(n-c)$-flat for some $1\le c\le b$. This means that the set of possible weights is $\{0, 2^n-2^{n-1}, 2^n-2^{n-2}, ..., 2^n-2^0, 2^n\}$.

To the author's knowledge, this problem has not yet been investigated for $r=2$ (corresponding to $k=3$). However the weights of the $\mathrm{RM}(n,2)$ codewords themselves are known: Sloane and Berlekamp \cite{Sloane_Berlekamp} gave the exact weight enumerator for $\mathrm{RM}(n,2)$ with the arising weights being $2^{n-1}$ and $2^{n-1}\pm 2^{n-1-j}$ for $0\le j\le \left\lfloor\frac{n}{2}\right\rfloor$. No such description is known for higher order, however there are some partial results \cite{McEliece72,KTA76,KLP12}.

\section{Obtaining $[k,1]$-avoiders from linear codes}\label{sec:avoiderfromcode}

As mentioned in the Introduction, we now show a way to explicitly construct a $[k,1]$-avoider from an arbitrary binary linear code $C$, where the size of the set can be obtained from evaluating the weight enumerator of the code at certain inputs.

\label{thm:restated_codebased}
\restateavoiderfromcode*

To prove this theorem, we build up a set $S=S_C(k)\subseteq \ff_2^n$ using the following construction:

\begin{construction}[Code-based construction]\label{con:codebased}

If $k\ge 3$ and $\ell\ge 1$ are integers, and $C\le \ff_2^{\ell}$ is a binary linear code, then we define a set $S_C(k)\subseteq \ff_2^n$ as follows, where $n=\ell(k-1)$.

Let $F_1, ..., F_{\ell}$ be the $(n-k+1)$-flats in $\ff_2^n$ given by
$$F_i=\{\mathbf{x}\in \ff_2^n: \mathbf{x}_j=1 ~\forall j: (i-1)(k-1)+1\le j\le i(k-1)\}$$
for each $1\le i\le \ell$.

For $1\le i\le \ell$, define the indicator functions $I_i: \ff_2^n\to \ff_2$ by 
\begin{equation*}
I_i(\mathbf{x})=
\begin{cases}
1 & \text{if $\mathbf{x}\in F_i$,}\\
0 & \text{if $\mathbf{x}\not\in F_i$.}\\
\end{cases}
\end{equation*}

Now define the \textit{signature} function $\mathbf{s}: \ff_2^n\to \ff_2^{\ell}$ by $\mathbf{s}(\mathbf{x})=(I_1(\mathbf{x}), ..., I_{\ell}(\mathbf{x}))$, and define the set $S_C(k)$ by
$$\mathbf{x}\in S_C(k) \Leftrightarrow \mathbf{s}(\mathbf{x})\not \in C.$$
\end{construction}

\begin{proof}[Proof of Theorem \ref{thm:avoiderfromcode}]
Consider the set $S=S_C(k)$ as defined in Construction \ref{con:codebased}.

First we determine the size of $S_C(k)$, which we will do by turning to the complement. For each codeword $\mathbf{c}=(c_1, ..., c_{\ell})\in C$, the number of elements of $\ff_2^n$ with signature $\mathbf{c}$ depends on $w(\mathbf{c})$. If it is known that $I_i(\mathbf{x})=c_i$ for each $1\le i\le \ell$, then for the values of $i$ with $c_i=1$ the corresponding $k-1$ coordinates of $\mathbf{c}$ must be all $1$ (so can take just one value), and for those with $c_i=0$ the corresponding $k-1$ coordinates can be anything except all $1$'s, so can take $2^{k-1}-1$ values. Overall,
$$|\mathbf{x}\in \ff_2^n: \mathbf{s}(\mathbf{x})=\mathbf{c}|=(2^{k-1}-1)^{\ell-w(\mathbf{c})}.$$

Altogether, we have
$$|\overline{S_C(k)}|=\sum_{\mathbf{c}\in C} (2^{k-1}-1)^{\ell-w(\mathbf{c})}=\sum_{w=0}^{\ell} A_w(2^{k-1}-1)^{\ell-w}=W_C(1, 2^{k-1}-1),$$

so $|S_C(k)|=2^n-W_C(1, 2^{k-1}-1)$ as required.

We now show that $S_C(k)$ is indeed a $[k,1]$-avoider. Let $H$ be a parity check matrix for $C$, which is an $(\ell-d)\times \ell$ matrix where $d=\dim C$. So for any $\mathbf{y}\in \ff_2^{\ell}$, $\mathbf{y}\in C \Leftrightarrow H\mathbf{y}^T=\mathbf{0}$, and so for $\mathbf{x}\in \ff_2^n$, $\mathbf{x}\not\in S_C(k) \Leftrightarrow H\mathbf{s(x)}^T=\mathbf{0}$.

Fix $1\le r\le \ell-d$, and let the set of indices $j$ with $H_{r,j}=1$ be $\{j_{r,1}, j_{r,2}, ..., j_{r,b_r}\}$, where $1\le b_r\le \ell$. A vector $\mathbf{x}\in \ff_2^n$ satisfies $(H\mathbf{s}(\mathbf{x})^T)_r=1$ if and only if $$1=\sum_{i=1}^{b_r} \mathbf{s}(\mathbf{x})_{j_{r,i}}=\sum_{i=1}^{b_r} I_{j_{r,i}}(\mathbf{x}).$$

This holds if and only if $\mathbf{x}$ lies in $T_r=\symdiff_{i=1}^{b_r} F_{j_{r,i}}$. So overall, $\mathbf{x}\in S_C(k) \Leftrightarrow \mathbf{x}\in \bigcup_{r=1}^{\ell-d} T_r$, so $S_C(k)$ is of the form described in Observation \ref{obs:symdiff_subspaces}(c). Thus it is indeed a  $[k,1]$-avoider.
\end{proof}

\begin{remark}
Instead of a binary \textit{linear} code, the statement of Theorem \ref{thm:avoiderfromcode} also holds if $C$ is a binary \textit{affine} code $C\subseteq \ff_2^{\ell}$, that is, $C$ is of the form $C=L+\mathbf{y}_0$ for some linear subspace $L\le \ff_2^{\ell}$ and $\mathbf{y}_0\in \ff_2^{\ell}$. In this case, some of the conditions $(H\mathbf{s}(\mathbf{x})^T)_r=1$ might be replaced with $(H\mathbf{s}(\mathbf{x})^T)_r=0$, which holds if $\mathbf{x}\in T'_r=T_r\triangle \ff_2^n$. Since $\ff_2^n$ is itself a symmetric difference (in particular, a disjoint union) of $(n-k+1)$-flats, Observation \ref{obs:symdiff_subspaces}(c) still applies.
\end{remark}

Now given the result of Theorem \ref{thm:avoiderfromcode}, the following question naturally arises: is it possible to use it to disprove Conjecture \ref{con:maincon} by setting $t=1$? In other words, can $W_C(1, 2^{k-1}-1)$ take $\Omega(2^n)$ distinct values for some pair $(k,\ell)$? In fact, the following two questions emerge, which might raise interest on their own right.

\begin{problem}
Determine the number of distinct weight enumerator polynomials $W_C(x,y)$ that a binary linear code $C$ of length $\ell$ can have. We denote this number by $N_{\ell}$.
\end{problem}

\begin{problem}
For a fixed pair of integers $(a,b)$, determine the number of possible values $W_C(a,b)$ that can arise for a binary linear code $C$ of length $\ell$. We denote this number by $N_{\ell}(a,b)$.
\end{problem}

The following elementary upper bound can be given for the latter question.

\begin{proposition}\label{prop:elementarybound}
Let $a,b\ne 0$ and $\ell\ge 1$ be integers, and $g=\gcd(a,b)$. Then

$$N_{\ell}(a,b)\le\begin{cases}
\left(\frac{|a|+|b|}{g}\right)^{\ell} & \text{if $ab>0$,} \\
2\left(\frac{|a|+|b|}{g}\right)^{\ell}+1 & \text{if $ab<0$.}
\end{cases}$$
\end{proposition}
\begin{proof}
Recall that $W_C(a,b)=\sum_{w=0}^{\ell} A_w a^w b^{\ell-w}$, where $A_w$ is the number of codewords of weight $w$ in $C$. Here $0\le A_w\le \binom{\ell}{w}$ for all $w>0$, and $A_0=1$.

If $a,b>0$, then $W_C(a,b)$ is an integer that is divisible by $g^{\ell}$ and satisfies 
$$1\le W_C(a,b)\le \sum_{w=0}^{\ell} \binom{\ell}{w} a^w b^{\ell-w}=(a+b)^{\ell},$$
giving the required bound. If $a,b<0$, then $W_C(a,b)=(-1)^{\ell}W_C(|a|,|b|)$, leading to the previous case.

In the case $ab<0$, $W_C(a,b)$ is still divisible by $g^{\ell}$, and 
$$|W_C(a,b)|\le \sum_{w=0}^{\ell} \binom{\ell}{w} |a|^w |b|^{\ell-w}=(|a|+|b|)^{\ell}$$
by the triangle inequality, leaving at most $2\left(\frac{|a|+|b|}{g}\right)^{\ell}+1$ possible values.
\end{proof}

Now we present a second way of giving an upper bound on $N_{\ell}(a,b)$ in the fashion of Proposition \ref{prop:elementarybound}, namely through the dual of the code. For some triples $(a,b,\ell)$ this results in a better bound.

\begin{proposition}\label{prop:dualbound}
Let $a,b,\ell$ be integers such that $b>|a|$ and $\ell\ge 1$. Then
$$N_{\ell}(a,b)\le (\ell+1)\cdot \left(\frac{2b}{\gcd(b-a,b+a)}\right)^{\ell}.$$
\end{proposition}
\begin{proof}
We use the MacWilliams identity (Lemma \ref{lem:macwilliams}): if the code $C\le \ff_2^{\ell}$ has dimension $d$, then 
$$W_C(a,b)=\frac{1}{2^{\ell-d}} W_{C^{\perp}}(b-a,b+a).$$

Since $0\le d\le \ell$, the first factor can take $\ell+1$ distinct values, and by Proposition \ref{prop:elementarybound}, the second factor can take at most $\left(\frac{2b}{\gcd(b-a,b+a)}\right)^{\ell}$ distinct values, giving the result.
\end{proof}

As we will now see, this shows that we cannot disprove Conjecture \ref{con:maincon} by solely using our code-based construction. In fact, we have the following stronger statement.

\begin{corollary}\label{cor:macwilliams_limit}
If $k\ge 3$ is a fixed integer and $k-1\mid n$, then as $n\to \infty$, the number of possible sizes of $[k,1]$-avoiders in $\ff_2^n$ obtainable using the code-based construction is $O(nc_k^n)$, where $c_k<2$ is a fixed constant depending on $k$.
\end{corollary}
\begin{proof}
Let us apply Proposition \ref{prop:dualbound} with $a=1$, $b=2^{k-1}-1$, and $\ell=\frac{n}{k-1}$. We get 
$$N_{\ell}(1,2^{k-1}-1)\le (\ell+1)(2^{k-1}-1)^{\ell},$$
so via Theorem \ref{thm:avoiderfromcode}, we get the same upper bound on the number of sizes of $[k,1]$-avoiders obtainable using the code-based construction.

Since $\ell=\frac{n}{k-1}$, this upper bound is $O(nc_k^n)$, where $c_k=(2^{k-1}-1)^{1/(k-1)}<2$.
\end{proof}

Note that in particular, $c_3=\sqrt{3}$, which corresponds to the case of $[3,1]$-avoiders.

Having seen this exponential upper bound on the number of possible values of $W_C(1, 2^{k-1}-1)$ (equivalently, on the number of sizes of $[k,1]$-avoiders obtained from the code-based construction), it is natural to ask to determine the asymptotic value of this number, or at least to prove a reasonable lower bound on it. In the next section, we focus on the case $k=3$, for which we give an argument leading to an exponential lower bound. Observe that as any $k$-flat for $k\ge 4$ can be partitioned into $3$-flats, any $[3,1]$-avoider is also a $[k,1]$-avoider for all $k\ge 4$, but our code-based construction explicitly gives further $[k,1]$-avoiders as well.

\section{$[3,1]$-avoiders of exponentially many sizes via the code-based construction}\label{sec:expmany31avoid}

Now we will use the code-based construction demonstrated in the previous section to obtain $[3,1]$-avoiders of exponentially many different sizes, whose corresponding codes $C\le \ff_2^{\ell}$ come from a systematically described family of codes. Hence we devote this section to proving Theorem \ref{thm:expmany}, which we recall here.

\restateexpmany*

First we prove a lemma showing how the weight enumerator of a code $C$ changes when certain transformations are applied to $C$. We introduce the following two transformations:

\begin{definition}[Base code transformations]\label{def:basecodetransformation}
Let $C\le \ff_2^L$ be a linear code. Then:

\vspace{-1mm}
\begin{itemize}
\itemsep0em
\item we let $\mathtt{a}(C)\le \ff_2^{L+2}$ be the linear code obtained from $C$ by appending two $0$ coordinates to the end of each codeword,
\item we let $\mathtt{b}(C)\le \ff_2^{L+2}$ be the linear code obtained as $\mathtt{b}(C)=\mathtt{a}(C)\oplus\left\langle (1,1,...,1,1,0), (1,1,...,1,0,1) \right\rangle$.
\end{itemize}
\end{definition}

\begin{lemma}\label{lem:transf}
 If $C\le \ff_2^L$ is a linear code, then:

\vspace{-1mm}
\begin{enumerate}[label=(\alph*)]
\itemsep0em
\item $W_{\mathtt{a}(C)}(x,y)=y^2W_C(x,y)$,
\item $W_{\mathtt{b}(C)}(x,y)=(x^2+y^2)W_C(x,y)+2xyW_C(y,x)$.
\end{enumerate}
\end{lemma}
\begin{proof}
\begin{enumerate}[label=(\alph*)]
\item The statement is immediate by noting that the transformation leaves the weight of each codeword intact but increases the length of the code by two.
\item For every codeword $\mathbf{c}\in C$, the four distinct codewords $(\mathbf{c},0,0)$, $(\mathbf{c}+\mathbf{e},0,1)$, $(\mathbf{c}+\mathbf{e},1,0)$ and $(\mathbf{c},1,1)$ appear in $\mathtt{b}(C)$, where $\mathbf{e}=(1,1,...,1)\in \ff_2^L$. Therefore
\begin{equation*}
\begin{split}
W_{\mathtt{b}(C)}(x,y) &=\sum_{\mathbf{c'}\in \mathtt{b}(C)} x^{w(\mathbf{c'})}y^{L+2-w(\mathbf{c'})} \\
&=\sum_{\mathbf{c}\in C} \left(x^{w(\mathbf{c})}y^{L+2-w(\mathbf{c})} + x^{L-w(\mathbf{c})+1} y^{w(\mathbf{c})+1} + x^{L-w(\mathbf{c})+1} y^{w(\mathbf{c})+1} + x^{w(\mathbf{c})+2}y^{L-w(\mathbf{c})}\right),
\end{split}
\end{equation*}
where the four terms of the latter sum correspond to the aforementioned four codewords obtained from $\mathbf{c}$. This gives
\[W_{\mathtt{b}(C)}(x,y)=y^2W_C(x,y)+xyW_C(y,x)+xyW_C(y,x)+x^2W_C(x,y). \qedhere\]
\end{enumerate}
\end{proof}

Figure \ref{fig:transformations} shows the effect of the transformations $\mathtt{a}$ and $\mathtt{b}$ on a generator matrix of a code $C$.

\begin{figure}[H]
\centering
\begin{minipage}{0.3\textwidth}
	\centering
	$G_C=\begin{pmatrix}
		1 & 0 & 0 & 1 \\
		0 & 1 & 0 & 1 \\
		1 & 1 & 1 & 0
	\end{pmatrix}$
\end{minipage}
\begin{minipage}{0.3\textwidth}
	\centering
	$G_{\mathtt{a}(C)}=\begin{pmatrix}
	1 & 0 & 0 & 1 & \textcolor{red}{0} & \textcolor{red}{0} \\
	0 & 1 & 0 & 1 & \textcolor{red}{0} & \textcolor{red}{0} \\
	1 & 1 & 1 & 0 & \textcolor{red}{0} & \textcolor{red}{0}
	\end{pmatrix}$
\end{minipage}
\begin{minipage}{0.3\textwidth}
	\centering
	$G_{\mathtt{b}(C)}=\begin{pmatrix}
	1 & 0 & 0 & 1 & \textcolor{red}{0} & \textcolor{red}{0} \\
	0 & 1 & 0 & 1 & \textcolor{red}{0} & \textcolor{red}{0} \\
	1 & 1 & 1 & 0 & \textcolor{red}{0} & \textcolor{red}{0} \\
	\textcolor{blue}{1} & \textcolor{blue}{1} & \textcolor{blue}{1} & \textcolor{blue}{1} & \textcolor{purple}{1} & \textcolor{purple}{0} \\
	\textcolor{blue}{1} & \textcolor{blue}{1} & \textcolor{blue}{1} & \textcolor{blue}{1} & \textcolor{purple}{0} & \textcolor{purple}{1}
	\end{pmatrix}$
\end{minipage}

\caption{An example for the generating matrices of a code $C$ and the transformed codes $\mathtt{a}(C)$ and $\mathtt{b}(C)$.}
\label{fig:transformations}	
\end{figure}

We are going to use various finite sequences formed from the transformations $\mathtt{a}$ and $\mathtt{b}$. So for a word $f\in \{\mathtt{a},\mathtt{b}\}^r$ (where $r\in \zz_{\ge 0}$), say $f=\sigma_1\sigma_2...\sigma_r$ where each $\sigma_i\in \{\mathtt{a},\mathtt{b}\}$, define $f(C)=(\sigma_1\circ \sigma_2\circ ...\circ \sigma_r)(C)$. Using Figure \ref{fig:transformations}, for $f\in \{\mathtt{a},\mathtt{b}\}^r$ and $C\le \ff_2^L$, it is straightforward to determine the generator matrix of the code $f(C)\le \ff_2^{L+2r}$. Since $\mathtt{a}$ leaves the dimension unchanged whereas $\mathtt{b}$ increases it by $2$, $\dim f(C)=\dim C+2n_{\mathtt{b}}$ for a word $f$ containing $n_{\mathtt{b}}$ letters $\mathtt{b}$.

Using Theorem \ref{thm:avoiderfromcode} for $k=3$, if $\ell\ge 1$ is fixed and $n=2\ell$, then we get a $[3,1]$-avoider in $\ff_2^n$ of size $2^n-W_C(1,3)$ for any linear code $C\le \ff_2^{\ell}$. For an integer $r\ge 1$, we can start from the trivial code $C_0\le\ff_2^0$, and we can apply a sequence $f$ of $r$ transformations of the form $\mathtt{a}$ or $\mathtt{b}$ as described above, getting a code $f(C_0)\le \ff_2^{2r}$. Using Lemma \ref{lem:transf}, it is natural to track the value of not just $W_C(1,3)$, but also $W_C(3,1)$ during these transformations. To do this, we define the following auxiliary vector for a code $C$.

\begin{notation}
If $C\le \ff_2^L$ is a linear code, let $\mathbf{v}_C=\twovec{W_C(1,3)}{W_C(3,1)}$.
\end{notation}

Using this notation, one can read off the following from Lemma \ref{lem:transf}.

\begin{observation}\label{obs:twomatr}
For any linear code $C\le \ff_2^L$, we have $\mathbf{v}_{\mathtt{a}(C)}=\begin{pmatrix}9 & 0\\0 & 1\end{pmatrix}\mathbf{v}_C$ and $\mathbf{v}_{\mathtt{b}(C)}=\begin{pmatrix}10 & 6\\6 & 10\end{pmatrix}\mathbf{v}_C$.
\end{observation}

Now let $M_{\mathtt{a}}=\begin{pmatrix}9 & 0\\0 & 1\end{pmatrix}$ and $M_{\mathtt{b}}=\begin{pmatrix}10 & 6\\6 & 10\end{pmatrix}$. Take the group homomorphism $\varphi: F(\{\mathtt{a},\mathtt{b}\})\to GL_2(\qq)$ induced by $\varphi(\mathtt{a})=M_{\mathtt{a}}$ and $\varphi(\mathtt{b})=M_{\mathtt{b}}$. 

For any $f\in \{\mathtt{a},\mathtt{b}\}^*$, use the notation $M_f:=\varphi(f)$. In particular, for a word $f=\sigma_1\sigma_2...\sigma_r\in \{\mathtt{a},\mathtt{b}\}^r$, we have $M_f=M_{\sigma_1}M_{\sigma_2}...M_{\sigma_r}$. It is a clear consequence of Observation \ref{obs:twomatr} that for a code $C\le \ff_2^L$, and $f\in \{\mathtt{a},\mathtt{b}\}^*$, we have $\mathbf{v}_{f(C)}=M_f\mathbf{v}_C$.

For the trivial code $C_0\le \ff_2^0$, we have $\mathbf{v}_{C_0}=\twovec{1}{1}$, and after applying the $r$ transformations, we get $\mathbf{v}_{f(C_0)}=M_f \twovec{1}{1}=\twovec{W_{f(C_0)}(1,3)}{W_{f(C_0)}(3,1)}$.

We will prove that the number of distinct values that the first entry of $M_f\twovec{1}{1}$ can take grows exponentially in $r$. For this, the main idea is to show that $M_f$ is distinct for each $f\in \{\mathtt{a},\mathtt{b}\}^r$, referring to the following well-known group-theoretical result. It was first used by Klein, and we state the modern formulation of de la Harpe \cite[pp. 25-26]{Harpe}.

\begin{lemma}[Ping-pong lemma \cite{Harpe}]\label{lem:pingpong}
 Let $G$ be a group acting on a set $X$, and let $\Gamma_1,\Gamma_2\le G$ with $|\Gamma_1|\ge 3$ and $|\Gamma_2|\ge 2$. Let $\Gamma=\left\langle \Gamma_1, \Gamma_2\right\rangle\le G$.
 Suppose that $\emptyset\ne X_1,X_2$ are two subsets of $X$ with $X_2\not\subseteq X_1$, satisfying
 \begin{enumerate}[label=(\arabic*)]
    \item $\gamma(X_2)\subseteq X_1$ for all $1\ne \gamma\in \Gamma_1$, and
    \item $\gamma(X_1)\subseteq X_2$ for all $1\ne \gamma\in \Gamma_2$.
 \end{enumerate}
 
 Then $\Gamma=\Gamma_1\ast \Gamma_2$, that is, $\Gamma$ is the free product of the two groups.
\end{lemma}

In \cite{Harpe}, a classical application of the ping-pong lemma is then presented, showing that $\twomatr{1}{2}{0}{1}$ and $\twomatr{1}{0}{2}{1}$ generate a free subgroup of $SL_2(\zz)$ of rank two. In $GL_2(\qq)$, we now adapt the same argument to our matrices $M_{\mathtt{a}}$ and $M_{\mathtt{b}}$.

\begin{proposition}\label{prop:freesubgroup}
The matrices $M_{\mathtt{a}}=\twomatr{9}{0}{0}{1}$ and $M_{\mathtt{b}}=\twomatr{10}{6}{6}{10}$ generate a free subgroup of $GL_2(\qq)$ of rank two.
\end{proposition}
\begin{proof}
We use Lemma \ref{lem:pingpong} for the standard action of $GL_2(\qq)$ on $\qq^2$, with $\Gamma_1=\langle M_{\mathtt{a}}\rangle$ and $\Gamma_2=\langle M_{\mathtt{b}}\rangle$ being infinite cyclic subgroups (since $\left|\det M_{\mathtt{a}}\right|,\left|\det M_{\mathtt{b}}\right|\ne 1$). Consider the following two subsets of $\qq^2$:

$X_1=\left\{\mathbf{v}=\twovec{x}{y}\in \qq^2: x,y\ne 0,~ \left(\frac{|y|}{|x|}<\frac13 \mathrm{~or~} \frac{|y|}{|x|}>3\right)\right\},$

$X_2=\left\{\mathbf{v}=\twovec{x}{y}\in \qq^2: x,y\ne 0,~  \frac13<\frac{|y|}{|x|}<3\right\}.$

The two sets are clearly disjoint and nonempty, so let us check conditions (1) and (2). Via simple calculations we can see the following.

\begin{enumerate}[label=(\arabic*)]
\item If $\mathbf{v}$ satisfies $\frac13<\frac{|y|}{|x|}<3$ then $M_{\mathtt{a}}^n\mathbf{v}=\twovec{x'}{y'}$ satisfies $\frac{|y'|}{|x'|}<\frac13$ for any $n\ge 1$, and $\frac{|y'|}{|x'|}>3$ for any $n\le -1$. In all cases we have $x',y'\ne 0$.
\item If $\mathbf{v}$ satisfies $0<\frac{|y|}{|x|}<\frac13$ or $\frac{|y|}{|x|}>3$ then $M_{\mathtt{b}}^n\mathbf{v}=\twovec{x'}{y'}$ satisfies $\frac13<\frac{y'}{x'}<3$ for any $n\ge 1$, and $-3<\frac{y'}{x'}<-\frac13$ for any $n\le -1$. In all cases we have $x',y'\ne 0$.
\end{enumerate}

We can hence conclude that $\Gamma=\Gamma_1\ast \Gamma_2$, i.e., $M_{\mathtt{a}}$ and $M_{\mathtt{b}}$ generate a free group of rank two.
\end{proof}

As $M_{\mathtt{a}}$ and $M_{\mathtt{b}}$ generate a free group, for any fixed $r$ the matrices $M_f$ for words of length $r$ are all distinct. We exploit this to prove the existence of a large number of words $f$ for which the first entry of $M_f\twovec{1}{1}$ is distinct, and ultimately, for given $\ell$, a large number of codes $C$ of length $\ell$ having distinct values of $W_C(1,3)$.

\begin{proof}[Finishing the proof of Theorem \ref{thm:expmany}]
	First let us assume $n=4n'$, in which case the code-based construction uses codewords of length $2n'$.
	
	By Proposition \ref{prop:freesubgroup}, for every $r\ge 1$ the set $\mathcal{M}_r=\{M_f: f\in \{\mathtt{a},\mathtt{b}\}^r\}$ contains $2^r$ distinct matrices $M_f$ of size $2\times 2$.
	
	Let us now restrict our attention to those $M_g\in \mathcal{M}_{n'-1}$ for which $g$ contains exactly $\lfloor (n'-1)/2\rfloor$ letters $\mathtt{a}$ and $\lceil (n'-1)/2\rceil$ letters $\mathtt{b}$: let $\mathcal{M}^*_{n'-1}$ be the set of these matrices. These $T$ matrices, where $T=\binom{n'-1}{\lfloor (n'-1)/2\rfloor}=\binom{\left\lfloor(n-4)/4\right\rfloor}{\left\lfloor(n-4)/8\right\rfloor}$, all have the same determinant, namely $D=9^{\lfloor (n'-1)/2\rfloor}\cdot 64^{\lceil (n'-1)/2\rceil}$.
	
	Introduce the notation $M_f=\twomatr{\alpha_f}{\beta_f}{\gamma_f}{\delta_f}$ for a word $f$. We will show that the first entry of $M_f\twovec{1}{1}$, which is $\alpha_f+\beta_f$, can take at least $T^{1/3}$ distinct values for $M_f\in \mathcal{M}_{n'}$, proving the theorem in this case.
	
	For $M_g\in \mathcal{M}^*_{n'-1}$, count the number of distinct values $\alpha_g+\beta_g$ takes, depending on which we distinguish two cases.
	
	 \begin{description}[style=nextline,leftmargin=1em]
	\item[Case 1: The values of $\alpha_g+\beta_g$ take at least $T^{1/3}$ distinct values.]
	In this case, consider the set $\mathtt{a}\{\mathtt{a},\mathtt{b}\}^{n'-1} = \{\mathtt{a}g: g\in \{\mathtt{a},\mathtt{b}\}^{n'-1}\}\subseteq \{\mathtt{a},\mathtt{b}\}^{n'}$. Then for $f=\mathtt{a}g$, $M_{f}=\twomatr{9}{0}{0}{1}M_g$, so $\alpha_{f}+\beta_{f}=9(\alpha_g+\beta_g)$ also takes at least $T^{1/3}$ distinct values.
	
	\item[Case 2: The values of $\alpha_g+\beta_g$ take less than $T^{1/3}$ distinct values.]
	Then by pigeonhole, there is a subset $\mathcal{M}'\subseteq \mathcal{M}^*_{n'-1}$ of size at least $T^{2/3}$ for which $\alpha_g+\beta_g$ takes identical values. For the matrices $M_g\in\mathcal{M'}$, count the number of distinct values $9\alpha_g+\beta_g$ takes, distinguishing two further cases.
	
	\begin{description}[style=nextline,leftmargin=1em]
	\item[Case 2a: The values of $9\alpha_g+\beta_g$ take at least $T^{1/3}$ distinct values.] In this case, consider the set $\{\mathtt{a},\mathtt{b}\}^{n'-1}\mathtt{a} = \{g\mathtt{a}: g\in \{\mathtt{a},\mathtt{b}\}^{n'-1}\}\subseteq \{\mathtt{a},\mathtt{b}\}^{n'}$. Then for $f=g\mathtt{a}$, $M_{f}=M_g\twomatr{9}{0}{0}{1}$, so $\alpha_{f}+\beta_{f}=9\alpha_g+\beta_g$ also takes at least $T^{1/3}$ distinct values.
	
	\item[Case 2b: The values of $9\alpha_g+\beta_g$ take less than $T^{1/3}$ distinct values.] Then by pigeonhole, there exists $\mathcal{M}''\subseteq \mathcal{M}'$ of size at least $T^{1/3}$ for which both $\alpha_g+\beta_g$ and $9\alpha_g+\beta_g$ are constant for $M_g\in \mathcal{M}''$, meaning that both $\alpha_g\equiv \alpha$ and $\beta_g\equiv \beta$ are constant. Then for all matrices $M_g\in \mathcal{M}''$, we have $\alpha\delta_g-\beta\gamma_g=D$, where $\alpha,\beta,\gamma_g,\delta_g\in \zz_{\ge 0}$ and $D\in \zz_{>0}$. One can easily see that all solutions of this equation must have distinct values of $\gamma_g+\delta_g$. Therefore, considering the set $\mathtt{b}\{\mathtt{a},\mathtt{b}\}^{n'-1}=\{\mathtt{b}g: g\in \{\mathtt{a},\mathtt{b}\}^{n'-1}\}\subseteq \{\mathtt{a},\mathtt{b}\}^{n'}$, for $f=\mathtt{b}g$ one has $M_{f}=\twomatr{10}{6}{6}{10}M_g$, so $\alpha_f+\beta_f=10(\alpha+\beta)+6(\gamma_g+\delta_g)$ are all distinct, taking at least $T^{1/3}$ distinct values.
		\end{description}
	\end{description}
	
	Now assume $n=4n'+2$ instead. Then by the previous argument, we can get at least $\binom{\left\lfloor(n-6)/4\right\rfloor}{\left\lfloor(n-6)/8\right\rfloor}^{1/3}=\binom{\left\lfloor(n-4)/4\right\rfloor}{\left\lfloor(n-4)/8\right\rfloor}^{1/3}$ codes of length $2n'$ having distinct values of $W_{C_{2n'}}(1,3)$. If we add a $0$ coordinate to the end of each codeword in such a code $C_{2n'}$, we get a code $C_{2n'+1}\le \ff_2^{2n'+1}$ with $W_{C_{2n'+1}}(1,3)=3W_{C_{2n'}}(1,3)$, hence attaining $[3,1]$-avoiders in $\ff_2^n$ of at least $\binom{\left\lfloor(n-4)/4\right\rfloor}{\left\lfloor(n-4)/8\right\rfloor}^{1/3}$ sizes from the code-based construction.
	
	At last, it is well-known that $\binom{s}{\left\lfloor s/2\right\rfloor}\sim \frac{2^s}{\sqrt{\frac{\pi}{2}\cdot s}}$ as $s\to \infty$, resulting in the given lower bound of $$\binom{\left\lfloor(n-4)/4\right\rfloor}{\left\lfloor(n-4)/8\right\rfloor}^{1/3}\ge c\cdot\frac{2^{n/12}}{\sqrt{n}}$$ for some absolute constant $c>0$.
\end{proof}
\begin{remark}
The author believes that because of the free group structure, the values of the first entry of $M_f\twovec{1}{1}$ for $f\in \{\mathtt{a},\mathtt{b}\}^r$ actually take at least $(2-o(1))^r$ distinct values as $r\to \infty$. If this is true, then this proof method would give at least $(2^{1/4}-o(1))^n$ different sizes of $[3,1]$-avoiders.
\end{remark}
\section*{Acknowledgement}
The author would like to thank Zoltán Lóránt Nagy for his helpful remarks regarding this manuscript.

\end{document}